\documentclass{amsart}
\RequirePackage{fix-cm}
\RequirePackage[l2tabu, orthodox]{nag}
\RequirePackage[all, warning]{onlyamsmath}

\usepackage{latexsym}
\usepackage{amsmath,amsfonts,amssymb}
\usepackage{amsthm}
\usepackage{graphicx,color,xcolor}
\usepackage{overpic}
\usepackage{subcaption}
\usepackage{braket}
\usepackage[bookmarks=true,bookmarksnumbered=true]{hyperref}
\hypersetup{
    colorlinks=true,
    citecolor=blue,
    linkcolor=blue,
    urlcolor=orange,
}
\usepackage{multirow,bigdelim}
\usepackage{enumerate}
\usepackage{cleveref}
\usepackage{comment}
\numberwithin{equation}{section}
\usepackage{doi}

\newtheorem{theorem}{Theorem}[section]
\newtheorem{lemma}[theorem]{Lemma}

\newtheorem{proposition}[theorem]{Proposition}

\newtheorem{definition}[theorem]{Definition}

\newtheorem{example}[theorem]{Example}
\newtheorem{remark}[theorem]{Remark}
\crefname{section}{Section}{Sections}
\crefname{appendix}{Appendix}{Appendices}
\crefname{theorem}{Theorem}{Theorems}
\crefname{lemma}{Lemma}{Lemmas}
\crefname{corollary}{Corollary}	{Corollaries}			
\crefname{proposition}{Proposition}{Propositions}	
\crefname{claim}{Claim}{Claims}
\crefname{conjecture}{Conjecture}{Conjectures}			
\crefname{definition}{Definition}{Definitions}
\crefname{problem}{Problem}{Problems}
\crefname{example}{Example}{Examples}
\crefname{remark}{Remark}{Remarks}
\crefname{figure}{Figure}{Figures}
\crefname{footnote}{Footnote}{Footnotes}
\crefname{equation}{}{}
\crefname{enumi}{}{}
\crefformat{enumi}{(#2#1#3)}
\Crefformat{enumi}{(#2#1#3)}

\newcommand{\QED}{\hfill \ensuremath{\Box}}
\newcommand{\R}{\mathbb{R}}

\newcommand{\ld}{,\ldots,}

\newcommand{\wt}{\widetilde}

\newcommand{\cv}{\mathrm{cv}}

\newcommand{\norm}[1]{\left\|#1\right\|}

\newcommand{\HD}{\mathrm{HD}}
\newfont{\bg}{cmr9 scaled\magstep2}
\newcommand{\bigzerol}{\smash{\lower1.0ex\hbox{\bg 0}}}
\DeclareMathOperator{\rank}{rank}
\DeclareMathOperator{\corank}{corank}
\DeclareMathOperator{\img}{Im}

\DeclareMathOperator{\codim}{codim}
\DeclareMathOperator{\id}{id}

\newcommand{\red}[1]{{\color{black} #1}}

\title[A refined transversality theorem on linear perturbations]{A refined transversality theorem on
\\ 
linear perturbations and its applications}
\author{Shunsuke Ichiki
}
\address{
Department of Mathematical and Computing Science,
School of Computing,
Tokyo Institute of Technology,
Tokyo 152-8552,
Japan}
\email{ichiki@c.titech.ac.jp}

\begin{document}
\date{}
%%%%%%%%%%%%%%%%%% abstract %%%%%%%%%%%%%%%%%%%%%
\begin{abstract}
In this paper, we establish a refined transversality theorem on linear perturbations from a new perspective of Hausdorff measures.
Furthermore, we give its applications not only to singularity theory but also to multiobjective optimization.
\end{abstract}
%58K25 Stability theory for manifolds
%58K30 Global theory of singularities
%57R35 Differentiable mappings
%57R45 Singularities of differentiable mappings
%%%%%%%%%%%%%%%%%%%%%%%%%%%%%%%%%%%%%%% 
\maketitle
\noindent

%%%%%%%%%%%%%%%%%%%%%%%%%%%%%%%%%%%%%%%%%%%%%%%%%%%%%%%%%%%%%%%%%%%%%%%%%%%%%%%%
%%%%%%%%%%%%%%%%% Introduction 
\section{Introduction}\label{sec:intro}
%%%%%%%%%%%%%%%%%%%%%%%%%%%%%%%%%%%%%%%

%Let $\pi:\R^m\to \R^\ell$, $V$ and $f:V\to \R^\ell$ be a linear mapping, an open subset of $\R^m$ and a mapping, respectively.
Transversality theorems are fundamental tools for investigating generic mappings.
In 1973, Mather gave a striking transversality theorem on generic projections as the main theorem of the celebrated paper \cite{Mather1973}.
%Briefly, Mather's result is as follows.
Let $\mathcal{L}(\R^{m},\R^{\ell})$ be the space consisting of all linear mappings of $\R^{m}$ into $\R^{\ell}$. 
In what follows, we will regard $\mathcal{L}(\R^m, \R^\ell)$ as the Euclidean space $(\R^m)^\ell$ in the obvious way.
Briefly, Mather's result is a transversality theorem for a composition $\pi\circ f:X\to \R^\ell$ of a $C^\infty$ embedding $f$ from a $C^\infty$ manifold $X$ into $\R^m$ and a projection $\pi\in \mathcal{L}(\R^m,\R^\ell)\setminus\Sigma$, where $\Sigma$ is a subset of $\mathcal{L}(\R^m,\R^\ell)$ with Lebesgue measure zero.
The theorem yields important applications on a composition of a $C^\infty$ embedding and a projection (e.g. Theorems 2 and 3 of \cite{Mather1973}).
%The main theorem of \cite{Mather1973} yields many applications.

After that, in \cite{Ichiki2018C}, for a $C^\infty$ immersion $f$ from a $C^\infty$ manifold $X$ into an open set $V$ of $\R^m$ and an arbitrary $C^\infty$ mapping $g:V\to \R^\ell$, 
a transversality theorem on the $1$-jet extension of a composition of $f$ and a mapping obtained by generically linearly perturbing $g$, that is $(g+\pi)\circ f:X\to \R^\ell$  $(\pi\in \mathcal{L}(\R^m,\R^\ell)\setminus\Sigma)$, is given, where $\Sigma$ is a subset of $\mathcal{L}(\R^m,\R^\ell)$ with Lebesgue measure zero.

Moreover, in \cite{Ichiki2018T}, the transversality theorem of \cite{Ichiki2018C} on generic linear perturbations described above had been improved so that it works even in the case where manifolds and mappings are not necessarily of class $C^\infty$ (see \cref{thm:appmain1_t} of this paper).
The theorem also yields applications on a composition of an immersion and a generically linearly perturbed mapping (as in \cref{thm:Morse2,thm:immersion2} of this paper).
Moreover, it also gives an application not only to singularity theory but also to multiobjective optimization (see \cref{thm:simplicial_generic} of this paper). 
Namely, it has been a useful tool to yield various applications on generically linearly perturbed mappings.

However, the transversality theorem (\cref{thm:appmain1_t}) is still in the stage of Lebesgue measures and therefore are its applications since Thom's parametric transversality theorem is used as a lemma in its proof.
On the other hand, in \cite{Ichiki2022a}, Thom's parametric transversality theorem had been already improved from a new perspective of Hausdorff measures which generalize Lebesgue measures.
Thus, the purpose of this paper is to give a refined version of the transversality theorem from the viewpoint of  Hausdorff measures and its various applications.

In \cite{Ichiki2018C,Ichiki2018T}, not only immersions but also injections are investigated.
More precisely, in \cite{Ichiki2018C}, for a $C^\infty$ injection $f$ from a $C^\infty$ manifold $X$ into an open set $V$ of $\R^m$ and a $C^\infty$ mapping $g:V\to \R^\ell$, a specialized transversality theorem on crossings of a composition of $f$ and a mapping obtained by generically linearly perturbing $g$ and its applications are also given \red{from the viewpoint of Lebesgue measures}.
Furthermore, in \cite{Ichiki2018T}, the specialized transversality theorem and some of its applications have been improved so that they work even in the case where manifolds and mappings are not necessarily of class $C^\infty$.
\red{
On the other hand, for refined versions of these injective cases from the viewpoint of Hausdorff measures, there are still some supplementary problems that remain unsolved (for details, see \cref{rem:thm:appmain1_rem}~\cref{rem:thm:appmain1_rem_injection}).
Thus, in this paper, we do not treat injective cases, but we give an application to multiobjective optimization from the viewpoint of singularity theory and differential topology, which is a refined version of a result obtained in \cite{Hamada2019} (\cref{thm:simplicial_generic} in this paper) from a new perspective of Hausdorff measures. 
}

The remainder of this paper is organized as follows.
In \cref{sec:main}, we state the main theorem.
In \cref{sec:preparation}, we review the definition of Hausdorff measures and prepare an essential tool for the proof of the main theorem, and in \cref{sec:appmain1}, we show the main theorem.
In \cref{sec:appfurther}, we give applications of the main theorem.
In \cref{sec:appmulti}, we also give an application of the main theorem to multiobjective optimization from the viewpoint of singularity theory and differential topology, and in \cref{sec:simpliciality_generic_proof}, we give the proof of the application.

\section{The main theorem}\label{sec:main}
In this paper, unless otherwise stated, all manifolds are without boundary and assumed to have a countable basis. 
In this section, we prepare some notations and state the main theorem.
%For the statement of the main theorem, we prepare some definitions and notations.
First, we review the definition of transversality. 
\begin{definition}\label{def:transversality}
{\rm 
Let $X$ and $Y$ be $C^r$ manifolds, and $Z$ a $C^r$ submanifold of $Y$ ($r\geq1$).
Let $f : X\to Y$ be a $C^1$ mapping. 
\begin{enumerate}[(1)]
\item 
We say that $f:X\to Y$ is \emph{transverse} to $Z$ \emph{at $x\in X$} if $f(x)\not\in Z$ or in the case of $f(x)\in Z$, the following holds: 
\begin{align*}
df_x(T_xX)+T_{f(x)}Z=T_{f(x)}Y.
\end{align*}
\item 
We say that $f:X\to Y$ is \emph{transverse} to $Z$ if for any $x\in X$, the mapping $f$ is transverse to $Z$ at $x$. 
\end{enumerate}
}
\end{definition}
Let $X$ be a $C^r$ manifold $(r\geq 2)$ of dimension $n$, and $J^1(X,\R^\ell)$ the space of $1$-jets of mappings of $X$ into $\R^\ell$. 
Then, note that $J^1(X,\R^\ell)$ is a $C^{r-1}$ manifold. 
For a given $C^r$ mapping $f:X\to \R^\ell$, the $1$-jet extension $j^1 f:X\to J^1(X,\R^\ell)$ is defined by $q \mapsto j^1 f(q)$. 
Then, notice that $j^1 f$ is of class $C^{r-1}$. 
For details on $J^1(X,\R^\ell)$ and $j^1 f$, see for example, \cite{Golubitsky1973}. 

Now, let $\set{(V_\lambda ,\varphi _\lambda )}_{\lambda \in \Lambda}$ be a coordinate neighborhood system of $X$. 
Let $\Pi :J^1(X,\R^\ell)\to X\times \R^\ell$ be the natural projection defined by $\Pi(j^1f(q))=(q,f(q))$.
Let $\Phi _\lambda :\Pi^{-1}(V_\lambda \times \R^\ell)\to \varphi _\lambda (V_\lambda)\times \R^\ell \times J^1(n,\ell)$ be the homeomorphism defined by 
\begin{align*} 
\Phi _\lambda \left(j^1f(q)\right)=\left(\varphi _\lambda (q),f(q),j^1(\psi_{_\lambda} \circ f\circ \varphi _\lambda ^{-1}\circ \wt{\varphi} _\lambda)(0)\right), 
\end{align*} 
where $J^1(n, \ell)=\set{ j^1f(0)| f : (\R^n, 0) \to  (\R^\ell, 0)}$ and $\wt{\varphi} _\lambda : \R^n\to \R^n$ (resp., $\psi_{\lambda} : \R^\ell \to \R^\ell$) is the parallel translation satisfying $\wt{\varphi} _\lambda(0)=\varphi _\lambda (q)$ (resp., $\psi_{\lambda}(f(q))=0$). 
Then, $\set{(\Pi^{-1}(V_\lambda \times \R^\ell), \Phi _\lambda )}_{\lambda \in \Lambda}$ is a coordinate neighborhood system of $J^1(X,\R^\ell)$. 

Set 
\begin{align*} 
S^k=\set{j^1f(0)\in J^1(n,\ell)|\corank Jf(0)=k}, 
\end{align*}
where $\corank Jf(0)=\min \set{n,\ell}-\rank Jf(0)$ and $k=1,2,\ldots ,\min \set{n, \ell}$. 
Set 
\begin{align*}
S^k(X,\R^\ell)=\bigcup_{\lambda \in \Lambda}\Phi ^{-1}_\lambda \left(\varphi _\lambda (V_\lambda )\times \R^\ell \times S^k\right). 
\end{align*}
Then, the set $S^k(X,\R^\ell)$ is a submanifold of $J^1(X,\R^\ell)$ satisfying  
\begin{align*}
\codim S^k(X,\R^\ell)&=\dim J^1(X,\R^\ell)- 
\dim S^k(X,\R^\ell)=(n-v+k)(\ell-v+k), 
\end{align*}
where $v=\min \set{n, \ell}$. 
(For details on $S^k$ and $S^k(X,\R^\ell)$, see 
\cite[pp.~60--61]{Golubitsky1973}). 
\begin{proposition}[\cite{Ichiki2018T}]\label{thm:appmain1_t} 
%%%
Let $f:X\to V$ be a $C^r$ immersion, and $g:V\to \R^\ell$ a $C^r$ mapping, where $r$ is an integer satisfying $r\geq 2$, $X$ is a $C^r$ manifold and $V$ is an open subset of $\R^m$.
Let $k$ be an integer satisfying $1\leq k\leq \min\set{\dim X,\ell}$. 
If 
\begin{align*}
r\geq\max\set{\dim X-\codim S^k(X,\R^\ell),0}+2,
\end{align*}
then the set
\begin{align*}
\Sigma_k=\set{\pi\in \mathcal{L}(\R^m,\R^\ell)|\mbox{$j^1((g+\pi)\circ f)$ is not transverse to $S^k(X,\R^\ell)$}}
\end{align*}
has Lebesgue measure zero in $\mathcal{L}(\R^m,\R^\ell)$.
\end{proposition}

As a side note, \cite[Theorem~1]{Ichiki2018C} is \cref{thm:appmain1_t} in the case where all manifolds and mappings are of class $C^\infty$.
Namely, \cref{thm:appmain1_t} is an improvement of \cite[Theorem~1]{Ichiki2018C}.
The following is the main theorem of this paper, which is a refined version of \cref{thm:appmain1_t} from a new perspective of Hausdorff measures.
%In this paper, as a main theorem, we give a refined version of \cref{thm:appmain1_t} from a new perspective of Hausdorff measures as follows:
\begin{theorem}\label{thm:appmain1} 
%%%
Let $f:X\to V$ be a $C^r$ immersion, and $g:V\to \R^\ell$ a $C^r$ mapping, where $r$ is an integer satisfying $r\geq 2$, $X$ is a $C^r$ manifold and $V$ is an open subset of $\R^m$.
Let $k$ be an integer satisfying $1\leq k\leq \min\set{\dim X,\ell}$. 
Set
\begin{align*}
\Sigma_k=\set{\pi\in \mathcal{L}(\R^m,\R^\ell)|\mbox{$j^1((g+\pi)\circ f)$ is not transverse to $S^k(X,\R^\ell)$}}.
\end{align*}
%where $k$ is an integer satisfying $1\leq k \leq \min \set{\dim X,\ell}$. 
Then, the following hold: 
\begin{enumerate}[$(1)$]
\item \label{thm:appmain1sub1}
Suppose $\dim X-\codim S^k(X,\R^\ell)\geq 0$.  
%If 
%\begin{align*}
%r\geq \dim X-\codim S^k(X,\R^\ell)+2,
%\end{align*}
Then, for any real number $s$ satisfying 
\begin{align}\label{eq:appmain1}
s\geq m\ell-1+\frac{\dim X-\codim S^k(X,\R^\ell)+1}{r-1}, 
\end{align}
the set $\Sigma_k$ has $s$-dimensional Hausdorff measure zero in $\mathcal{L}(\R^{m},\R^{\ell})$.
\item \label{thm:appmain1sub2}
Suppose $\dim X-\codim S^k(X,\R^\ell)<0$. 
Then, we have the following:
\begin{enumerate}[{\rm (2a)}]
\item \label{thm:appmain1sub2_h}%[$(2a)$]
For any real number $s$ satisfying 
\begin{align}\label{eq:appmain1_second}
s>m\ell+\dim X-\codim S^k(X,\R^\ell), 
\end{align}
the set $\Sigma_k$ has $s$-dimensional Hausdorff measure zero in $\mathcal{L}(\R^{m},\R^{\ell})$.\label{c}
\item \label{thm:appmain1sub2_e}%[$(2b)$]
For any $\pi\in \mathcal{L}(\R^{m},\R^{\ell})\setminus\Sigma_k$, we have $ j^1((g+\pi)\circ f)(X) \cap S^k(X,\R^\ell)=\varnothing$.
\end{enumerate}
\end{enumerate}
%%%
\end{theorem}

\begin{remark}\label{rem:thm:appmain1_rem}
{\rm We give the following remarks on \cref{thm:appmain1}.
\begin{enumerate}[(1)]
\item \label{rem:thm:appmain1_rem_improvement}
We will show that \cref{thm:appmain1} implies \cref{thm:appmain1_t}.
Let $f$ and $g$ (resp. $k$ and $r$) be mappings (resp.  integers) satisfying the assumption of \cref{thm:appmain1_t}.
First, we consider the case $\dim X-\codim S^k(X,\R^\ell)\geq 0$.
Since $r\geq \dim X-\codim S^k(X,\R^\ell)+2$, we can set $s=m\ell$ in \cref{eq:appmain1}.
Thus, since $\Sigma_k$ has $m\ell$-dimensional Hausdorff measure zero in $\mathcal{L}(\R^{m},\R^{\ell})$ by \cref{thm:appmain1}~\cref{thm:appmain1sub1}, $\Sigma_k$ also has Lebesgue measure zero.
In the case $\dim X-\codim S^k(X,\R^\ell)<0$, since we can set $s=m\ell$ in \cref{eq:appmain1_second}, $\Sigma_k$ has $m\ell$-dimensional Hausdorff measure zero in $\mathcal{L}(\R^{m},\R^{\ell})$ by \cref{thm:appmain1}~\cref{thm:appmain1sub2_h}, which implies that $\Sigma_k$ has Lebesgue measure zero.

\item \label{rem:thm:appmain1_rem_infty} 
In \cref{thm:appmain1}~\cref{thm:appmain1sub1}, if all manifolds and mappings are of class $C^\infty$, then for any real number $s$ such that $s>m\ell-1$, there exists a positive integer $r$ satisfying \cref{eq:appmain1}.
Thus, in the $C^\infty$ case, we can replace \cref{eq:appmain1} by
\begin{align*}%\label{eq:s_infinity}
    s>m\ell-1.
\end{align*}

\item
In \cref{thm:appmain1}, since $f$ is an immersion, we have $n\leq m$, where $n=\dim X$. 
Thus, in \cref{thm:appmain1}~\cref{thm:appmain1sub2_h}, since  
\begin{align*}
m\ell+\dim X-\codim S^k(X,\R^\ell)
\geq m\ell+n-n\ell
=(m-n)\ell+n
\geq n,
\end{align*}
it is not necessary to assume that $s$ is non-negative.

\item
In \cref{thm:appmain1}, there is an advantage that the domain of $g:V\to \R^\ell$ is not $\R^m$ but an arbitrary open subset $V$ of $\R^m$. 
Suppose $V=\R$. 
Let $g:\R\to \R$ be the function defined by $g(x)=|x|$. 
Since $g$ is not differentiable at $x=0$, we cannot apply \cref{thm:appmain1} to $g:\R\to \R$. 
On the other hand, if $V=\R\setminus\set{0}$, then we can apply the theorem to $g|_V$. 
\item
\red{The assumptions \cref{eq:appmain1} and \cref{eq:appmain1_second} cannot be improved in general (see \cref{rem:main_evaluation_1} and \cref{rem:main_evaluation_2}, respectively), which implies that these are the best evaluations in general.}
\item \label{rem:thm:appmain1_rem_injection}
\red{As explained in \cref{sec:intro}, in \cite{Ichiki2018C,Ichiki2018T}, not only the case where $f$ is an immersion but also the case where $f$ is an injection is investigated from the viewpoint of Lebesgue measures.
By using the refined version of Thom's parametric transversality theorem obtained in \cite{Ichiki2022a} (\cref{thm:main} in this paper), we can certainly update some results on injections obtained in \cite{Ichiki2018C,Ichiki2018T} from the new perspective of Hausdorff measures.
However, for those updated results, it is unsolved whether the evaluations on Hausdorff measures (such as \cref{eq:appmain1} and \cref{eq:appmain1_second} in immersion cases) are the best or not in general.
Thus, in this paper, we do not deal with injective cases, but we give an application of \cref{thm:appmain1} to multiobjective optimization in \cref{sec:appmulti,sec:simpliciality_generic_proof}.
}
\end{enumerate}}
\end{remark}

\section{Preliminaries for the proof of the main theorem}\label{sec:preparation}
First, we review the definition of Hausdorff measures.
Let $s$ be an arbitrary non-negative real number.
Then, the \emph{$s$-dimensional Hausdorff outer measure} on $\R^n$ is defined as follows.
Let $B$ be a subset of $\R^n$.
The $0$-dimensional Hausdorff outer measure of $B$ is the number of points in $B$.
For $s>0$, the $s$-dimensional Hausdorff outer measure of $B$ is defined by 
\begin{align*}
   \lim_{\delta \to 0} \mathcal{H}_{\delta}^s(B),
\end{align*}
where for each $0<\delta \leq \infty$, 
\begin{align*}
\mathcal{H}_{\delta}^s(B)=\inf\Set{\sum_{j=1}^\infty (\mbox{diam}\,C_j)^s|B\subset \bigcup_{j=1}^\infty C_j,\  \mbox{diam}\,C_j\leq \delta}.
\end{align*}
Here, for a subset $C$ of $\R^n$, we write 
\begin{align*}
\textrm{diam}\ C=\sup \set{\|x-y\||x, y\in C}, 
\end{align*}
where $\|z\|$ denotes the Euclidean norm of $z\in \R^n$.
Note that the infimum in $\mathcal{H}_{\delta}^s(B)$ is over all coverings of $B$ by subsets $C_1, C_2,\ldots$ of $\R^n$ satisfying $\textrm{diam}\ C_j\leq \delta$ for all positive integers $j$.

Let $s$ be an arbitrary non-negative real number.
Let $N$ be a $C^r$ manifold $(r\geq 1)$ of dimension $n$, and $\set{(U_\lambda, \varphi_\lambda)}_{\lambda\in \Lambda}$ a coordinate neighborhood system of $N$.
Then, a subset $\Sigma$ of $N$ has \emph{$s$-dimensional Hausdorff measure zero} in $N$ if for any $\lambda\in \Lambda$, the set $\varphi_\lambda(\Sigma \cap U_\lambda)$ has $s$-dimensional Hausdorff (outer) measure zero in $\R^n$.
Note that this definition does not depend on the choice of a coordinate neighborhood system of $N$.
Moreover, for a subset $\Sigma$ of $N$, set
%{\small
\begin{align*}
    \HD_N(\Sigma)=\inf\Set{s\in[0,\infty)|\mbox{$\Sigma$ has $s$-dimensional Hausdorff measure zero in $N$}},
\end{align*}
%}
which is called the \emph{Hausdorff dimension of $\Sigma$} in $N$.

Next, we will prepare an essential tool (\cref{thm:main}) for the proof of the main theorem, which is a refined version of Thom’s parametric transversality theorem and its improvement which was given by Mather in \cite{Mather1973}.
In order to state \cref{thm:main}, we prepare some definitions.
Let $X$, $A$ and $Y$ be $C^r$ manifolds ($r\geq1$), and $U$ an open set of $X\times A$.
In what follows, by $\pi_1 : U\to X$ and $\pi_2 : U\to A$, we denote the natural projections defined by 
\begin{align*}
\pi_1(x,a)=x,\ \ \pi_2(x,a)=a. 
\end{align*} 
Let $F:U\to Y$ be a $C^1$ mapping. 
For any element $a\in \pi_2(U)$, let 
\begin{align*}
F_a:\pi_1(U\cap (X\times \set{a})) \to Y 
\end{align*}
be the mapping defined by $F_a(x)=F(x,a)$. 
Here, note that $\pi_1(U\cap (X\times \set{a}))$ is open in $X$.
Let $Z$ be a submanifold of $Y$.
Set 
\begin{align*}
\Sigma(F,Z)=\set{a\in \pi_2(U)|\mbox{$F_a$ is not transverse to $Z$}}.
\end{align*}

\begin{definition}\label{def:delta}{\rm Let $X$ and $Y$ be $C^r$ manifolds, and $Z$ a $C^r$ submanifold of $Y$ ($r\geq1$).
Let $f : X\to Y$ be a $C^1$ mapping. 
For any $x\in X$, set  
\begin{align*}
\delta(f, x ,Z)&=\left\{ \begin{array}{ll}
0 & \mbox{if $f(x)\not \in Z$}, \\
\dim Y-\dim (df_x(T_xX)+T_{f(x)}Z) & \mbox{if $f(x)\in Z$}, \\
\end{array} \right.
\\ 
\delta(f,Z)&=\sup \set{\delta(f, x ,Z)|x\in X}. 
\end{align*}
}
\end{definition}
In the case that all manifolds and mappings are of class $C^\infty$, \cref{def:delta} is the definition of \cite[p.~230]{Mather1973}. 
As in \cite{Bruce2000}, $\delta(f,x,Z)$ measures the extent to which $f$ fails to be transverse to $Z$ at $x$.

\begin{definition}\label{def:W}{\rm 
Let $X$, $A$ and $Y$ be $C^r$ manifolds, and $Z$ a $C^r$ submanifold of $Y$ ($r\geq1$).
Let $F:U\to Y$ be a $C^1$ mapping, where $U$ is an open set of $X\times A$. 
Then, we define  
\begin{align*}
W(F,Z)&=\set{(x,a)\in U \mid \delta(F_a, x, Z)=\delta(F, (x,a), Z)>0}, 
\\
\delta^*(F,Z)&=\dim X+\dim Z-\dim Y +\delta(F,Z)
\\
&=\dim X-\codim Z+\delta(F,Z), 
\end{align*} 
where $\codim Z=\dim Y-\dim Z$.
}
\end{definition}
In what follows, we denote the image of a given mapping $f$ by $\img f$.
\begin{theorem}[\cite{Ichiki2022a}]\label{thm:main}
Let $X$, $A$ and $Y$ be $C^r$ manifolds, $Z$ a $C^r$ submanifold of $Y$, and $F:U\to Y$ a $C^r$ mapping, where $U$ is an open set of $X\times A$ and $r$ is a positive integer. 
Then, the following hold: 
\begin{enumerate}[$(1)$]
\item \label{thm:main1}
Suppose $\delta^*(F,Z)\geq 0$.
Then, for any real number $s$ satisfying 
\begin{align}\label{eq:main}
s\geq \dim A-1+\frac{\delta^*(F,Z)+1}{r}, 
\end{align}
the following $(\alpha)$ and $(\beta)$ are equivalent. 
\begin{enumerate}%[$(\alpha)$]
\item [$(\alpha)$]
The set $\pi_2(W(F,Z))$ has $s$-dimensional Hausdorff measure zero in $\pi_2(U)$. 
\item [$(\beta)$] 
The set $\Sigma(F,Z)$ has $s$-dimensional Hausdorff measure zero in $\pi_2(U)$.
\end{enumerate}
\item \label{thm:main2}
Suppose $\delta^*(F,Z)<0$.
Then, the following hold:
\begin{enumerate}[\upshape(2a)]
\item \label{thm:main2_w}
We have $W(F,Z)=\varnothing$.
\item \label{thm:main2_h}
For any non-negative real number $s$ satisfying $s>\dim A+\delta^*(F,Z)$, the set $\Sigma(F,Z)$ has $s$-dimensional Hausdorff measure zero in $\pi_2(U)$.
\item \label{thm:main2_e}
For any $a\in \pi_2(U)\setminus\Sigma(F,Z)$, we have $\img F_a \cap Z=\varnothing$.
\end{enumerate}
\end{enumerate}
%%%
\end{theorem}

%%%%%%%%%%%%%%%%%%%%%%%%%%%%%%%%%%%%%%%%%%%%%%%%%%
\section{Proof of the main theorem}\label{sec:appmain1}
%%%%%%%%%%%%%%%%
%%%%%%%%%%%%%%%%%%%%%%%%%%%%%%%%%%
\red{
%The method of this proof is almost the same as those of \cite[Theorem~1]{Ichiki2018C} and \cref{thm:appmain1_t}.
Let $\Gamma:X\times \mathcal{L}(\R^{m},\R^{\ell}) \to J^1(X,\R^\ell)$ be the $C^{r-1}$ mapping defined by 
\begin{align*}
\Gamma(q,\pi)=j^1((g+\pi) \circ f)(q). 
\end{align*}
The strategy of this proof is to apply \cref{thm:main} as $F=\Gamma$ and $Z=S^k(X,\R^\ell)$.
First, by the same method as in the proofs of \cite[Theorem~1]{Ichiki2018C} and \cref{thm:appmain1_t}, we can obtain $\delta(\Gamma, S^k(X,\R^\ell))=0$.}
\red{Thus, we have 
\begin{align*}%\label{eq:delta1}
\delta^*(\Gamma, S^k(X,\R^\ell))=\dim X-\codim S^k(X,\R^\ell). 
\end{align*}
Then, note that $\Sigma_k=\Sigma(\Gamma, S^k(X,\R^\ell))$. }
%\begin{align*}%\label{eq:delta1}
%\Sigma_k=\Sigma(\Gamma, S^k(X,\R^\ell)). 
%\end{align*}

Next, we will show \cref{thm:appmain1}~\cref{thm:appmain1sub1}. 
Since $\dim X-\codim S^k(X,\R^\ell)\geq 0$, we obtain $\delta^*(\Gamma, S^k(X,\R^\ell))\geq 0$. 
Notice that $\Gamma$ is of class $C^{r-1}$ $(r\geq 2)$ and we have 
\begin{align*}
%r-1&\geq \dim X-\codim S^k(X,\R^\ell)+1
%\\
%&=\delta^*(\Gamma, S^k(X,\R^\ell))+1, 
%\\
s&\geq m\ell-1+\frac{\dim X-\codim S^k(X,\R^\ell)+1}{r-1}
=m\ell-1+\frac{\delta^*(\Gamma, S^k(X,\R^\ell))+1}{r-1}. 
%\\
%&=m\ell-1+\frac{\delta^*(\Gamma, S^k(X,\R^\ell))+1}{r-1}. 
\end{align*}
Since $\delta(\Gamma, S^k(X,\R^\ell))=0$, we obtain $W(\Gamma, S^k(X,\R^\ell))=\varnothing$.
Therefore, the set $\pi_2(W(\Gamma, S^k(X,\R^\ell)))$ has $s$-dimensional Hausdorff measure zero in $\mathcal{L}(\R^{m},\R^{\ell})$. 
Thus, by  \cref{thm:main}~\cref{thm:main1}, the set $\Sigma(\Gamma, S^k(X,\R^\ell))$ has $s$-dimensional Hausdorff measure zero in $\mathcal{L}(\R^{m},\R^{\ell})$.
Since $\Sigma_k=\Sigma(\Gamma, S^k(X,\R^\ell))$, we have \cref{thm:appmain1}~\cref{thm:appmain1sub1}.

%there exists an $s$-dimensional Hausdorff measure zero set $\Sigma$ of $\mathcal{L}(\R^{m},\R^{\ell})$ such that the mapping $j^1((g+\pi) \circ f):X\to J^1(X,\R^\ell)$ is transverse to $S^k(X,\R^\ell)$ for any $\pi \in \mathcal{L}(\R^{m},\R^{\ell})-\Sigma$.

Finally, we will show \cref{thm:appmain1}~$\cref{thm:appmain1sub2}$. 
Since $\dim X-\codim S^k(X,\R^\ell)<0$, we obtain $\delta^*(\Gamma, S^k(X,\R^\ell))<0$. 
Since $\Gamma$ is of class $C^{r-1}$ $(r\geq 2)$ and we have 
\begin{align*}
s>m\ell+\dim X-\codim S^k(X,\R^\ell)=m\ell+\delta^*(\Gamma, S^k(X,\R^\ell)),
\end{align*}
the set $\Sigma(\Gamma, S^k(X,\R^\ell))$ has $s$-dimensional Hausdorff measure zero in $\mathcal{L}(\R^{m},\R^{\ell})$ by \cref{thm:main}~\cref{thm:main2_h}.
By \cref{thm:main}~\cref{thm:main2_e}, for any $\pi\in \mathcal{L}(\R^{m},\R^{\ell})\setminus\Sigma(\Gamma, S^k(X,\R^\ell))$, we have $\img \Gamma_\pi \cap S^k(X,\R^\ell)=\varnothing$.
Since $\Sigma_k=\Sigma(\Gamma, S^k(X,\R^\ell))$, we obtain \cref{thm:appmain1}~\cref{thm:appmain1sub2}.
\QED

%%%%%%%%%%%%%%%%%%%%%%%%%%%%%%%%%%%%%%%%%%%%%%%%%%%%%%%

%%%%%%%%%%%%%%%%%%%%%%%%%%%%%%%%%%%%%%%%%%%%%%%%%%    
\section{Applications of the main theorem}\label{sec:appfurther}
%%%%%%%%%%%
%%%%%%%%%%%%%%%%%%%%%%%%%%%%%%%%%%
In this section, we give applications and their examples of the main theorem in the two cases $\ell=1$ and $\ell\geq 2\dim X$.

First, we consider the case $\ell=1$.
Let $X$ be a $C^r$ manifold $(r\geq 1)$, and let $f : X\to \R$ be a $C^1$ mapping.
A point $x\in X$ is called a \emph{critical point} of $f$ if $\rank df_x=0$.
We say that a point of $\R$ is a \emph{critical value} if it is the image of a critical point.
A $C^r$ function $f:X\to \R$ $(r\geq 2)$ is called a \emph{Morse function} if all of the critical points of $f$ are nondegenerate, where $X$ is a $C^r$ manifold.
For details on Morse functions, see for example, \cite[p.~63]{Golubitsky1973}. 
In \cite{Ichiki2018T}, the following result is obtained as an application of \cref{thm:appmain1_t}.
%As a corollary of \cref{thm:Morse}, we can easily obtain the following result, which is Corollary~1 of \cite{Ichiki2018T} in the case $r=2$, and Corollary~1 of \cite{Ichiki2018C} in the case where all manifolds and mappings are of class $C^\infty$. 
%As a corollary of \cref{thm:Morse}, we can easily obtain the following result, which is Corollary~1 of \cite{Ichiki2018T} in the case $r=2$, and Corollary~1 of \cite{Ichiki2018C} in the case where all manifolds and mappings are of class $C^\infty$. 
%The following is Corollary~1 of \cite{Ichiki2018T}.
\begin{proposition}[\cite{Ichiki2018T}]\label{thm:Morse2}
Let $f$ be a $C^r$ immersion of a $C^r$ manifold $X$ into an open subset $V$ of $\R^m$, and $g:V\to \R$ a $C^r$ function, where $r$ is an integer satisfying $r\geq 2$. 
Then, the set 
\begin{align*}
\Sigma=\set{\pi\in \mathcal{L}(\R^m,\R)|\mbox{$(g+\pi) \circ f:X\to \R$ is not a Morse function}}
\end{align*}
has Lebesgue measure zero in $\mathcal{L}(\R^m,\R)$.
\end{proposition}
As a side note, Corollary~1 of \cite{Ichiki2018C} is \cref{thm:Morse2} in the case where all manifolds and mappings are of class $C^\infty$.
Namely, \cref{thm:Morse2} is an improvement of Corollary~1 of \cite{Ichiki2018C}.
In this paper, by using the main theorem, we further upgrade \cref{thm:Morse2} from a new perspective of Hausdorff measures as follows:
\begin{theorem}\label{thm:Morse}
Let $f$ be a $C^r$ immersion of a $C^r$ manifold $X$ into an open subset $V$ of $\R^m$, and $g:V\to \R$ a $C^r$ function, where $r$ is an integer satisfying $r\geq 2$. 
Set
\begin{align*}%\label{eq:Morse_bad}
\Sigma=\set{\pi\in \mathcal{L}(\R^m,\R)|\mbox{$(g+\pi) \circ f:X\to \R$ is not a Morse function}}.
\end{align*}
Then, for any real number $s$ satisfying 
\begin{align}\label{eq:Morse}
s\geq m-1+\frac{1}{r-1}, 
\end{align}
the set $\Sigma$ has $s$-dimensional Hausdorff measure zero in $\mathcal{L}(\R^m,\R)$.
\end{theorem}
\begin{remark}
{\rm
In \cref{thm:Morse}, if all manifolds and mappings are of class $C^\infty$, then we can replace \cref{eq:Morse} by $s>m-1$ by the same argument as in \cref{rem:thm:appmain1_rem}~\cref{rem:thm:appmain1_rem_infty} .
}
\end{remark}
\begin{proof}[Proof of \cref{thm:Morse}]
It is clearly seen that $\Sigma$ is the set consisting of all elements $\pi\in \mathcal{L}(\R^m,\R)$ satisfying that $j^1((g+\pi)\circ f)$ is not transverse to $S^1(X,\R)$.
Since $\dim X-\codim S^1(X,\R)=0$, by \cref{thm:appmain1}~\cref{thm:appmain1sub1}, for any real number $s$ satisfying \cref{eq:Morse}, the set $\Sigma$ has $s$-dimensional Hausdorff measure zero in $\mathcal{L}(\R^m,\R)$.
\end{proof}
As in the following example (\cref{ex:Morse}), there exists an example such that \cref{eq:Morse} in \cref{thm:Morse} cannot be improved.
Namely, \cref{eq:Morse} is the best evaluation in general.
In \cref{ex:Morse}, we also explain an advantage of \cref{thm:Morse} from a new perspective of Hausdorff measures compared to \cref{thm:Morse2} from the viewpoint of Lebesgue measures.

\begin{example}[An example of \cref{thm:Morse}]\label{ex:Morse}{\rm  
Set $X=V=\R$ and $f(x)=x$ in \cref{thm:Morse}.
%As in \cref{thm:Morse}, set 
%\begin{align*}
%\Sigma=\set{\pi\in \mathcal{L}(\R,\R)|\mbox{$g+\pi:\R\to \R$ is not a Morse function}}.
%\end{align*}
Let $g:\R\to \R$ be the $C^r$ function defined by $g(x)=-\int_{0}^{x}\eta(x)dx$, where $r$ is an integer satisfying $r\geq 2$ and $\eta:\R\to \R$ is a $C^{r-1}$ function such that the Hausdorff dimension of the set consisting of all critical values of $\eta$ is $\frac{1}{r-1}$.
Note that the existence of such a function is guaranteed  in \cite[Example~4.2]{Moreira2001}\footnote{In \cite[Example~4.2]{Moreira2001}, in general, there exists a $C^r$ mapping $\wt{\eta}:\R^m\to \R^m$ such that the Hausdorff dimension of the set consisting of all critical values of $\wt{\eta}$ (i.e. the set of all $\wt{\eta}(x)\in\R^m$ such that $x\in \R^m$ satisfies $\rank d\wt{\eta}_x<m$) is equal to $m-1+\frac{1}{r}$, where $r$ is a positive integer.}.
Set 
\begin{align*}
\Sigma=\set{a\in \R|\mbox{$g+\pi:\R\to \R$ is not a Morse function, where $\pi(x)=ax$}}.
\end{align*}
By noting that a linear mapping $\pi\in \mathcal{L}(\R,\R)$ can be expressed by $\pi(x)=ax$ $(a\in \R)$ and thus that it can be identified with $a\in \R$, for any real number $s$ satisfying
\begin{align}\label{eq:eta}
s\geq \frac{1}{r-1},
\end{align}
the set $\Sigma$ has $s$-dimensional Hausdorff measure zero in $\R$ by \cref{thm:Morse}.
Now, we will show that \cref{eq:eta} cannot be improved.
Let $\cv(\eta)$ be the set consisting of all critical values of $\eta$.
Since $a\in \Sigma$ if and only if there exists $x\in \R$ such that
$\frac{d(g+\pi)}{dx}(x)=0$ and $\frac{d^2(g+\pi)}{dx^2}(x)=0$ (i.e. $a=\eta(x)$ and $\frac{d\eta}{dx}(x)=0$) if and only if $a\in \cv(\eta)$, we obtain $\Sigma=\cv(\eta)$, which implies that $\HD_\R(\Sigma)=\frac{1}{r-1}$.
Namely, we cannot improve the assumption \cref{eq:eta}, which implies that \cref{eq:Morse} is the best evaluation in general.
%}
%    \end{example}
%\begin{example}[An example of \cref{thm:Morse}]
%{\rm 

Now, by using this example, we simply explain an advantage of \cref{thm:Morse} compared to \cref{thm:Morse2}.
%For a subset $\Sigma$ of $\Rb$, we denote the Hausdorff dimension of $\Sigma$ in $\R$ by $\HD_{\R}(\Sigma)$.
By the above argument, in the case $r\geq 3$, we have
\begin{align*}
\HD_{\R}(\Sigma) <\HD_{\R}(K)=\frac{\log 2}{\log 3}=0.63\cdots,
\end{align*}
where $K$ is the Cantor set in $\R$.
Thus, in the case $r\geq 3$. the set $\Sigma$ is never equal to $K$.

On the other hand, in the case $r=2$, there exists an example such that the bad set is equal to $K$ as follows.
Let $h:\R\to \R$ be the $C^2$ function defined by $h(x)=-\int_{0}^{x}\xi(x)dx$, where $\xi:\R\to \R$ is a $C^1$ function such that the set consisting of all critical values of $\xi$ is the Cantor set $K$.
Note that the existence of the mapping $\xi$ can be easily shown from \cite[Proposition~2 (p.~1485)]{Norton1991}\footnote{Proposition~2 of \cite{Norton1991} is as follows: If $\wt{K}$ is a compact subset of the closed interval [0,1], then $\wt{K}$ has Lebesgue measure zero if and only if the set consisting of all critical values of $\wt{\xi}$ is equal to $\wt{K}$ for some $C^1$ function $\wt{\xi}:\R\to\R$.
Since the Cantor set $K$ is a compact subset of $[0,1]$ with Lebesgue measure zero, we can guarantee the existence of the above function $\xi:\R\to \R$.}.
Set 
\begin{align*}
\Sigma'=\set{a\in \R|\mbox{$h+\pi:\R\to \R$ is not a Morse function, where $\pi(x)=ax$}}.
\end{align*}
Since we have $a\in \Sigma'$ if and only if there exists $x\in \R$ such that
$\frac{d(h+\pi)}{dx}(x)=0$ and $\frac{d^2(h+\pi)}{dx^2}(x)=0$ (i.e. $a=\xi(x)$ and $\frac{d\xi}{dx}(x)=0$), which is equivalent to that $a\in K$, we obtain $\Sigma'=K$. 

Since any subset of $\R$ whose Hausdorff dimension is strictly smaller than $1$ has Lebesgue measure zero in $\R$, we cannot investigate whether the bad set is equal to $K$ or not by \cref{thm:Morse2}.
On the other hand, as in the case $r\geq 3$ of this example, we can see that the bad set $\Sigma$ is never equal to $K$ by \cref{thm:Morse}.}
\end{example}
\red{Here, we give a remark on the assumption \cref{eq:appmain1} of \cref{thm:appmain1}. 
\begin{remark}\label{rem:main_evaluation_1}
{\rm  
    In \cref{thm:appmain1}, let $X=V=\R$, and let $f$ and $g$ be the functions defined in \cref{ex:Morse}.
    Then, $\Sigma_1$ in \cref{thm:appmain1} is expressed as follows:
\begin{align*}%\label{eq:sigma_1}
\Sigma_1=\set{\pi\in \mathcal{L}(\R,\R)|\mbox{$j^1(g+\pi)$ is not transverse to $S^1(\R,\R)$}}.
\end{align*}
    Since $\dim \R-\codim S^1(\R,\R)=0$, for any real number $s$ satisfying $s\geq \frac{1}{r-1}$, 
the set $\Sigma_1$ has $s$-dimensional Hausdorff measure zero in $\mathcal{L}(\R,\R)$ by \cref{thm:appmain1}~\cref{thm:appmain1sub1}.
Since $\Sigma_1$ can be identified with the set $\Sigma$ in \cref{ex:Morse}, we have $\HD_{\mathcal{L}(\R,\R)}(\Sigma_1)=\frac{1}{r-1}$.
Thus, the assumption $s\geq \frac{1}{r-1}$ cannot be improved, which implies that \cref{eq:appmain1} is the best evaluation in general.}
\end{remark}
}
Next, we consider the case $\ell \geq2\dim X$.
%As a corollary of \cref{thm:immersion}, we can easily obtain the following result, which is Corollary~2 of \cite{Ichiki2018T} in the case $r=2$, and Corollary~3 of \cite{Ichiki2018C} in the case where all manifolds and mappings are of class $C^\infty$. 
In \cite{Ichiki2018T}, the following result is also obtained as an application of \cref{thm:appmain1_t}.
%The following is Corollary~2 of \cite{Ichiki2018T}.
\begin{proposition}[\cite{Ichiki2018T}]\label{thm:immersion2}
Let $f$ be a $C^r$ immersion of an $n$-dimensional $C^r$ manifold $X$ into an open subset $V$ of $\R^{m}$, and $g:V\to \R^{\ell}$ a $C^r$ mapping, where $\ell \geq2n$ and $r\geq 2$. 
Then, the following set 
\begin{align*}
\Sigma=\set{\pi\in \mathcal{L}(\R^m,\R^\ell)|\mbox{$(g+\pi) \circ f:X\to \R^\ell$ is not an immersion}}
\end{align*}
has Lebesgue measure zero in $\mathcal{L}(\R^m,\R^\ell)$.
\end{proposition}
As a side note, Corollary~3 of \cite{Ichiki2018C} is \cref{thm:immersion2} in the case where all manifolds and mappings are of class $C^\infty$.
Namely, \cref{thm:immersion2} is an improvement of Corollary~3 of \cite{Ichiki2018C}.
In this paper, by using the main theorem, we further upgrade \cref{thm:immersion2} from a new perspective of Hausdorff measures as follows:
\begin{theorem}\label{thm:immersion}
Let $f$ be a $C^r$ immersion of an $n$-dimensional $C^r$ manifold $X$ into an open subset $V$ of $\R^{m}$, and $g:V\to \R^{\ell}$ a $C^r$ mapping, where $\ell \geq2n$ and $r\geq 2$. 
Set
\begin{align*}
\Sigma=\set{\pi\in \mathcal{L}(\R^m,\R^\ell)|\mbox{$(g+\pi) \circ f:X\to \R^\ell$ is not an immersion}}.
\end{align*}
Then, for any real number $s$ satisfying 
\begin{align}\label{eq:immersion}
s>m\ell+(2n-\ell-1), 
\end{align}
the set $\Sigma$ has $s$-dimensional Hausdorff measure zero in $\mathcal{L}(\R^m,\R^\ell)$.
\end{theorem}
%Now, we show \cref{thm:immersion}.
\begin{proof}[Proof of \cref{thm:immersion}]
Let $k$ be an integer satisfying $1\leq k \leq n$.
As in \cref{thm:appmain1}, set 
\begin{align*}
\Sigma_k=\set{\pi\in \mathcal{L}(\R^m,\R^\ell)|\mbox{$j^1((g+\pi)\circ f)$ is not transverse to $S^k(X,\R^\ell)$}}.
\end{align*}
Since $\ell\geq 2n$, we have 
\begin{align*}
\dim X-\codim S^k(X,\R^\ell) 
\leq \dim X-\codim S^1(X,\R^\ell) 
=2n-\ell-1< 0. 
\end{align*}
Hence, since 
\begin{align*}
s>m\ell+(2n-\ell-1)\geq m\ell+(\dim X-\codim S^k(X,\R^\ell)),  
\end{align*}
by \cref{thm:appmain1}~\cref{thm:appmain1sub2}, we have the following:
\begin{enumerate}[(a)]
    \item 
    The set $\Sigma_k$ has $s$-dimensional Hausdorff measure zero in $\mathcal{L}(\R^{m},\R^\ell)$.
    \item
    For any $\pi\in \mathcal{L}(\R^{m},\R^\ell)\setminus\Sigma_k$, we have 
    $j^1((g+\pi)\circ f)(X)\cap S^k(X,\R^\ell)=\varnothing$.
\end{enumerate}

By (b), we can easily obtain $\Sigma=\bigcup_{k=1}^n\Sigma_k$.
By (a), the set $\Sigma$ has $s$-dimensional Hausdorff measure zero in $\mathcal{L}(\R^{m},\R^\ell)$.
\end{proof}
As in the following example (\cref{ex:immersion}), there exists an example such that \cref{eq:immersion} in \cref{thm:immersion} cannot be improved.
Namely, \cref{eq:immersion} is the best evaluation in general.
In \cref{ex:immersion}, we also explain an advantage of \cref{thm:immersion} from a new perspective of Hausdorff measures compared to \cref{thm:immersion2} from the viewpoint of Lebesgue measures.
%As in the following example, in general,  \cref{eq:immersion} cannot be improved:
\begin{example}[An example of \cref{thm:immersion}]\label{ex:immersion}{\rm 
Set $X=V=\R$ and $f(x)=x$ in \cref{thm:immersion}.
Let $g:\R\to \R^\ell$ $(\ell\geq 2)$ be the $C^\infty$ mapping defined by $g(x)=(x^2\ld x^2)$.
As in \cref{thm:immersion}, set 
\begin{align*}
\Sigma=\set{\pi\in \mathcal{L}(\R,\R^\ell)|\mbox{$g+\pi:\R\to \R^\ell$ is not an immersion}}.
\end{align*}
Then, for any real number $s$ satisfying $s>1$, the set $\Sigma$ has $s$-dimensional Hausdorff measure zero in $\mathcal{L}(\R,\R^\ell)$ by \cref{thm:immersion}.
On the other hand, by the following direct calculation, we obtain $\Sigma=B$, where 
\begin{align*}
    B=\set{\pi=(\pi_1\ld \pi_\ell)\in\mathcal{L}(\R,\R^\ell)|\pi_1=\cdots=\pi_\ell}.
\end{align*}
Since $B$ does not have $1$-dimensional Hausdorff measure zero in $\mathcal{L}(\R,\R^\ell)$, we cannot improve the assumption $s>1$, which means that in general,  \cref{eq:immersion} cannot be improved.

%For example, we consider the case 
%Moreover, in the case $\ell\geq 3$, since a set with both of a straight line and a plane have Lebesgue measure zero in $\mathcal{L}(\R,\R^\ell)$ $(=\R^\ell)$, it is difficult to investigate whether the bad set $\Sigma$

Now, we prove $\Sigma=B$.
First, we show $\Sigma\subset B$.
Let $\pi=(\pi_1\ld \pi_\ell)\in \Sigma$ be an arbitrary element.
Set $\pi_i(x)=a_ix$ ($a_i\in\R$) for $i=1\ld \ell$.
Then, there exists $\wt{x}\in \R$ such that $2\wt{x}+a_i=0$ for all $i=1\ld \ell$.
Since $a_1=\cdots =a_\ell$, we have $\pi\in B$.

Next, we show $B\subset \Sigma$.
Let $\pi\in B$ be an arbitrary element.
Then, we can express $\pi_i(x)=ax$ ($a\in\R$) for all $i=1\ld \ell$.
Set $\wt{x}=-\frac{a}{2}$.
Since $d(g+\pi)_{\wt{x}}=0$, we obtain $\pi\in \Sigma$.

Now, by using this example, we explain an advantage of \cref{thm:immersion} compared to \cref{thm:immersion2}.
%Now, by using this example, we explain an advantage of \cref{thm:immersion} from the viewpoint of Hausdorff measures compared to \cref{thm:immersion2} from the viewpoint of Lebesgue measures.
%We consider the case $\ell\geq 3$. Since a plane in $\mathcal{L}(\R,\R^\ell)$ has  
%Any subset $B'$ of $\mathcal{L}(\R,\R^\ell)$ satisfying $\HD_{\mathcal{L}(\R,\R^\ell)}B'<\ell$ has Lebesgue measure zero.
Since any subset of $\mathcal{L}(\R,\R^\ell)$ whose Hausdorff dimension is strictly smaller than $\ell$ has Lebesgue measure zero in $\mathcal{L}(\R,\R^\ell)$, we cannot estimate the Hausdorff dimension of the bad set $\Sigma$ by \cref{thm:immersion2}.
On the other hand, by \cref{thm:immersion}, we have an estimate of the Hausdorff dimension of the bad set, such as $\HD_{\mathcal{L}(\R,\R^\ell)}(\Sigma)\leq 1$.
For example, we consider the case of $\ell\geq 3$.
Since a ``surface" such as a 2-dimensional sphere has Lebesgue measure zero in $\mathcal{L}(\R,\R^\ell)$, we cannot exclude the possibility that the bad set is such a 2-dimensional set by \cref{thm:immersion2}.
On the other hand, by using \cref{thm:immersion}, we can conclude that the bad set is never equal to a ``surface" such as a 2-dimensional sphere, since $\HD_{\mathcal{L}(\R,\R^\ell)}(\Sigma)\leq 1$.}
\end{example}
\red{Here, we give a remark on the assumption \cref{eq:appmain1_second} of \cref{thm:appmain1}. 
\begin{remark}\label{rem:main_evaluation_2}
{\rm  
    In \cref{thm:appmain1}, let $X=V=\R$ and let $f$ and $g:\R\to \R^\ell$ $(\ell \geq 2)$ be the mappings defined in \cref{ex:immersion}.
    Then, $\Sigma_1$ in \cref{thm:appmain1} is expressed as follows:
\begin{align*}%\label{eq:sigma_1}
\Sigma_1=\set{\pi\in \mathcal{L}(\R,\R^\ell)|\mbox{$j^1(g+\pi)$ is not transverse to $S^1(\R,\R^\ell)$}}.
\end{align*}
Since $\dim \R-\codim S^1(\R,\R^\ell)=1-\ell<0$, for any real number $s$ satisfying $s>1$, the set $\Sigma_1$ has $s$-dimensional Hausdorff measure zero in $\mathcal{L}(\R,\R^\ell)$ by \cref{thm:appmain1}~\cref{thm:appmain1sub2_h}.
Since $\Sigma_1$ is equal to the set $\Sigma$ in \cref{ex:immersion}, we cannot improve the assumption $s>1$, which implies that \cref{eq:appmain1_second} is the best evaluation in general.}
\end{remark}
}

%%%%%%%%%%%%%%%%%%%%%%%%
\section{An application of the main theorem to multiobjective optimization}\label{sec:appmulti}
%%%%%%%%%%%%%%%%%%%%%%%%
The purpose of this section is to give an application of the main theorem to multiobjective optimization from the viewpoint of differential topology and singularity theory  (see \cref{thm:simplicial_generic_2}) .
For a positive integer $\ell$, set 
\begin{align*}
    L=\set{1\ld \ell}.
\end{align*}

We consider the problem of optimizing several functions simultaneously.
More precisely, let $f: X \to \R^\ell$ be a mapping, where $X$ is a given arbitrary set.
A point $x \in X$ is called a \emph{Pareto solution} of $f$ if there does not exist another point $y \in X$ such that $f_i(y) \leq f_i(x)$ for all $i \in L$ and $f_j(y) < f_j(x)$ for at least one index $j \in L$.
We denote the set consisting of all Pareto solutions of $f$ by $X^*(f)$, which is called the \emph{Pareto set} of $f$.
The set $f(X^*(f))$ is called the \emph{Pareto front} of $f$.
The problem of determining $X^*(f)$ is called the \emph{problem of minimizing $f$}.

Let $f = (f_1\ld f_\ell): X \to \R^\ell$ be a mapping, where $X$ is a given arbitrary set.
For a non-empty subset $I = \set{i_1\ld i_k}$ of $L$ such that $i_1 < \dots < i_k$, set
\begin{align*}
    f_I = (f_{i_1}\ld f_{i_k}).
\end{align*}
The problem of determining $X^*(f_I)$ is called a \emph{subproblem} of the problem of minimizing $f$.
%In order to define the notion of (weak) simpliciality, we prepare some notions. 
Set
\begin{align*}
\Delta^{\ell-1}&=\Set{(w_1,\dots,w_\ell) \in \R^\ell | \sum_{i=1}^\ell w_i=1,\ w_i\geq 0}.
\end{align*}
We also denote a face of $\Delta^{\ell-1}$ for a non-empty subset $I$ of $L$ by
\begin{align*}
\Delta_I = \set{(w_1,\dots,w_\ell) \in \Delta^{\ell-1} | w_i = 0\ (i \not \in I)}.
\end{align*}

In this section, for a $C^r$ manifold $N$ (possibly with corners) and a subset $V$ of $\R^\ell$, a mapping $g:N\to V$ is called a \emph{$C^r$ mapping} (resp., a  \emph{$C^r$ diffeomorphism}) if $g:N\to \R^\ell$ is of class $C^r$ (resp., if $g:N\to \R^\ell$ is a $C^r$ immersion and $g:N\to V$ is a homeomorphism), where $r$ is a positive integer or $r=\infty$.
Here, $C^0$ mappings and $C^0$ diffeomorphisms are continuous mappings and homeomorphisms, respectively.
%
%By referring to \cite{Hamada2019}, we give the definition of (weakly) simplicial problems in this paper.
\begin{definition}[\cite{Hamada2019,Hamada2019b}]\label{def:simplicial}{\rm
Let $f = (f_1\ld f_\ell): X \to \R^\ell$ be a mapping, where $X$ is a subset of $\R^m$.
Let $r$ be an integer satisfying $r\geq 0$ or $r=\infty$.
The problem of minimizing $f$ is $C^r$ \emph{simplicial} if there exists a $C^r$ mapping $\Phi: \Delta^{\ell-1} \to X^*(f)$ such that both the mappings $\Phi|_{\Delta_I}: \Delta_I \to X^*(f_I)$ and $f|_{X^*(f_I)}: X^*(f_I) \to f(X^*(f_I))$ are $C^r$ diffeomorphisms for any non-empty subset $I$ of $L$.
The problem of minimizing $f$ is $C^r$ \emph{weakly simplicial} if there exists a $C^r$ mapping $\phi:\Delta^{\ell-1}\to X^*(f)$ such that $\phi(\Delta_I) = X^*(f_I)$ for any non-empty subset $I$ of $L$.}
\end{definition}
As described in \cite{Hamada2019}, simpliciality is an important property, which can be seen in several practical problems ranging from the facility location problem studied half a century ago~\cite{Kuhn1967} to sparse modeling actively developed today~\cite{Hamada2019}. If a problem is simplicial, then we can efficiently compute a parametric-surface approximation of the entire Pareto set with few sample points~\cite{Kobayashi2019}.

A subset $X$ of $\R^m$ is \emph{convex} if $t x + (1 - t) y \in X$ for all $x, y \in X$ and all $t \in [0, 1]$.
Let $X$ be a convex set in $\R^m$.
A function $f: X \to \R$ is \emph{strongly convex} if there exists $\alpha > 0$ such that
\begin{align*}
    f(t x + (1 - t) y) \leq t f(x) + (1 - t) f(y) - \frac{1}{2} \alpha t (1 - t) \norm{x - y}^2
\end{align*}
for all $x, y \in X$ and all $t \in [0, 1]$, where $\norm{z}$ is the Euclidean norm of $z \in \R^m$.
The constant $\alpha$ is called a \emph{convexity parameter} of the function $f$.
A mapping $f = (f_1\ld f_\ell): X \to \R^\ell$ is \emph{strongly convex} if $f_i$ is strongly convex for any $i \in L$.
%The problem of minimizing a strongly convex $C^r$ mapping is called the \emph{strongly convex $C^r$ problem}.

\begin{theorem}[\cite{Hamada2019,Hamada2019b}]\label{thm:main-C1}
Let $f:\R^m \to \R^\ell$ be a strongly convex $C^r$ mapping, where $r$ is a positive integer or $r=\infty$.
Then, the problem of minimizing $f$ is $C^{r-1}$ weakly simplicial.
Moreover, this problem is $C^{r-1}$ simplicial if the rank of the differential $df_x$ is equal to $\ell-1$ for any $x \in X^*(f)$.
\end{theorem}
%As a side note, for the assertion on simpliciality (resp., weak simpliciality) of \cref{thm:main-C1} in the case $r\geq 2$, see \cite[Theorem~1.1]{Hamada2019} (resp., \cite[Theorem~5]{Hamada2019b}).
%For the case $r=1$ of \cref{thm:main-C1}, see \cite[Theorem~2]{Hamada2019b}. 

Moreover, in \cite{Hamada2019}, the following result is obtained as an application of \cref{thm:appmain1_t}.
%Here, note that strong convexity is preserved under linear perturbations (see \cref{thm:preserve_strong} in \cref{sec:simpliciality_generic_proof}).

\begin{proposition}[\cite{Hamada2019}]\label{thm:simplicial_generic}
Let $f:\R^m\to\R^\ell$ $(m\geq \ell)$ be a strongly convex $C^r$ mapping, where $r$ is an integer satisfying $r\geq 2$ or $r=\infty$.
Set 
\begin{align*}
\Sigma=\set{\pi\in \mathcal{L}(\R^m,\R^\ell)|\mbox{The problem of minimizing $f+\pi$ is not $C^{r-1}$ simplicial}}.
\end{align*}
If $m-2\ell+4>0$, then $\Sigma$ has Lebesgue measure zero in $\mathcal{L}(\R^m, \R^\ell)$.
\end{proposition}
%As an improvement of \cref{thm:simplicial_generic}, we give the following application of \cref{thm:appmain1} (and hence \cref{thm:main}) to multiobjective optimization from the viewpoint of differential topology and singularity theory.
%, which is an improvement of \cref{thm:simplicial_generic} from the viewpoint of Hausdorff measures.
By using the main theorem, we can also upgrade \cref{thm:simplicial_generic} as follows:
\begin{theorem}\label{thm:simplicial_generic_2}
Let $f:\R^m\to\R^\ell$ $(m\geq \ell)$ be a strongly convex $C^r$ mapping, where $r$ is an integer satisfying $r\geq 2$ or $r=\infty$.
Set 
\begin{align*}
\Sigma=\set{\pi\in \mathcal{L}(\R^m,\R^\ell)|\mbox{The problem of minimizing $f+\pi$ is not $C^{r-1}$ simplicial}}.
\end{align*}
If $m-2\ell+4>0$, then for any non-negative real number $s$ satisfying 
\begin{align}\label{eq:generic_s_main}
s>m\ell-(m-2\ell+4),
\end{align}
the set $\Sigma$ has $s$-dimensional Hausdorff measure zero in $\mathcal{L}(\R^m,\R^\ell)$. 
\end{theorem}
%\begin{remark}{\rm
%There is an example of \cref{thm:simplicial_generic_2} such that we cannot improve the inequality \cref{eq:generic_s_main} (see \cref{ex:simplicial_generic}).}
%\end{remark}
As in the following example (\cref{ex:simplicial_generic}), there exists an example such that \cref{eq:generic_s_main} in \cref{thm:simplicial_generic_2} cannot be improved.
Namely, \cref{eq:generic_s_main} is the best evaluation in general.
In \cref{ex:simplicial_generic}, we also explain an advantage of \cref{thm:simplicial_generic_2} from a new perspective of Hausdorff measures compared to \cref{thm:simplicial_generic} from the viewpoint of Lebesgue measures.
%Now, we give an example on \cref{thm:simplicial_generic_2} (see \cref{ex:simplicial_generic}).
Now, in order to show that a given mapping in \cref{ex:simplicial_generic} is strongly convex, we prepare \cref{thm:norm_strong}, which is a well-known result (for the proof, for example, see \cite{Hamada2019b}).
Let $X$ be a convex subset of $\R^m$.
A function $f: X \to \R$ is said to be \emph{convex} if
\begin{align*}
    f(t x + (1 - t) y) \leq t f(x) + (1 - t) f(y)
\end{align*}
for all $x, y \in X$ and all $t \in [0, 1]$.
\begin{lemma}
\label{thm:norm_strong}
Let $X$ be a convex subset of $\R^m$.
Then, a function $f: X \to \R$ is strongly convex with a convexity parameter $\alpha > 0$ if and only if the function $g: X \to \R$ defined by $g(x) = f(x) - \frac{\alpha}{2} \norm{x}^2$ is convex.
\end{lemma}
\begin{example}[An example of \cref{thm:simplicial_generic_2}]\label{ex:simplicial_generic}{\rm 
Let $f=(f_1,f_2):\R^2\to \R^2$ be the mapping defined by 
$f_i(x_1,x_2)=x_1^2+x_2^2$ for $i=1,2$.
Since $g(x)=f_i(x)-\frac{2}{2} \norm{x}^2=0$ is convex, $f$ is strongly convex by \cref{thm:norm_strong}, where $x=(x_1,x_2)$.
As in \cref{thm:simplicial_generic_2}, set 
\begin{align*}
\Sigma=\set{\pi\in \mathcal{L}(\R^2,\R^2)|\mbox{The problem of minimizing $f+\pi$ is not $C^\infty$ simplicial}}.
\end{align*}
Then, for any real number $s$ satisfying $s>2$, the set $\Sigma$ has $s$-dimensional Hausdorff measure zero in $\mathcal{L}(\R^2,\R^2)$ by \cref{thm:simplicial_generic_2}. 

On the other hand, by the following direct calculation, we obtain $\Sigma=B$, where
\begin{align*}
B=\set{\pi=(\pi_1,\pi_2)\in \mathcal{L}(\R^2, \R^2)|\pi_1=\pi_2}.
\end{align*}
Since $B$ does not have $2$-dimensional Hausdorff measure zero in $\mathcal{L}(\R^2,\R^2)$, we cannot improve the assumption $s>2$.

Now, we show $\Sigma=B$.
First, in order to show that $\Sigma\subset B$, we will show that $\mathcal{L}(\R^2, \R^2)\setminus B\subset \mathcal{L}(\R^2, \R^2)\setminus\Sigma$.
%Next, we will show that $\mathcal{L}(\R^2,\R^2)-\Sigma\subset \mathcal{L}(\R^2,\R^2)-A$.
Let $\pi \in\mathcal{L}(\R^2,\R^2)\setminus B$ be an arbitrary element.
Let $H:\R^2\to \R^2$ be the diffeomorphism defined by $H(X_1,X_2)=(X_1-X_2,X_2)$.
As $f_1=f_2$, we obtain $H\circ (f+\pi)=(\pi_1-\pi_2,f_2+\pi_2)$.
Since $\pi_1-\pi_2$ is a linear function satisfying $\pi_1-\pi_2\not=0$, it follows that $\rank d(H\circ (f+\pi))_x\geq 1$ for any $x\in \R^2$.
As $H$ is a diffeomorphism, we have that $\rank d(f+\pi)_x\geq 1$ for any $x\in \R^2$.
By \cref{thm:main-C1}, the problem of minimizing $f+\pi$ is $C^\infty$ simplicial.
Namely, we obtain $\pi\in \mathcal{L}(\R^2,\R^2)\setminus\Sigma$.
%, where 
%Since $a_{11}\not=a_{21}$ or $a_{12}\not=a_{22}$, it is clearly seen that $\rank d(f+\pi)_x\geq 1$ for any $x\in \R^2$.

Next, we will show that $B\subset\Sigma$.
Let $\pi=(\pi_1,\pi_2)\in B$ be an arbitrary element.
Set $\pi_1(x_1,x_2)=\pi_2(x_1,x_2)=a_{1}x_1+a_{2}x_2$, where $a_{1},a_{2}\in \R$.
Since 
\begin{align*}
(f_i+\pi_i)(x_1,x_2)&=x_1^2+x_2^2+a_{1}x_1+a_{2}x_2
\\
&=\left(x_1+\frac{a_{1}}{2}\right)^2+\left(x_2+\frac{a_{2}}{2}\right)^2-\frac{a_{1}^2+a_{2}^2}{4}
\end{align*}
for $i=1,2$, we obtain $X^*(f+\pi)=\Set{(-\frac{a_{1}}{2},-\frac{a_{2}}{2})}$ $(\subset \R^2)$.
Hence, the problem of minimizing $f+\pi$ is not $C^0$ simplicial (and hence, not $C^\infty$ simplicial).
Namely, we obtain $\pi\in \Sigma$.

Finally, by using this example, we explain an advantage of \cref{thm:simplicial_generic_2}  compared to \cref{thm:simplicial_generic}.
%We consider the case $\ell\geq 3$. Since a plane in $\mathcal{L}(\R,\R^\ell)$ has  
%Any subset $B'$ of $\mathcal{L}(\R,\R^\ell)$ satisfying $\HD_{\mathcal{L}(\R,\R^\ell)}B'<\ell$ has Lebesgue measure zero.
%Since any subset of $\mathcal{L}(\R^2,\R^2)$ whose Hausdorff dimension is less than $4$ has Lebesgue measure zero in $\mathcal{L}(\R^2,\R^2)$, we cannot estimate the Hausdorff dimension of the bad set $\Sigma$ by \cref{thm:simplicial_generic}.
%On the other hand, by \cref{thm:simplicial_generic_2}, we can estimate the Hausdorff dimension of the bad set, such as $\HD_{\mathcal{L}(\R^2,\R^2)}(\Sigma)\leq 2$.
%For example, we consider the case of $\ell\geq 3$.
Since a set whose Hausdorff dimension is equal to $3$, such as a $3$-dimensional sphere, has Lebesgue measure zero in $\mathcal{L}(\R^2,\R^2)$, we cannot exclude the possibility that the bad set $\Sigma$ is such a ``$3$-dimensional set" by \cref{thm:simplicial_generic}.
On the other hand, by using \cref{thm:simplicial_generic_2}, we can conclude that $\Sigma$ is never equal to  such a ``$3$-dimensional set" since $\HD_{\mathcal{L}(\R^2,\R^2)}(\Sigma)\leq 2$.

}
\end{example}

%%%%%%%%%%%%%%%%%%%%%%%%%%%%%%%%%%%%%%%%%%%%%%%%%%%%%%%%%%%%%%%%%%%%%%%%%%%%%%%%
\section{Proof of \cref{thm:simplicial_generic_2}}\label{sec:simpliciality_generic_proof}
%%%%%%%%%%%%%%%%%%%%%%%%%%%%%%%%%%%%%%%%%%%%%%%%%%%%%%%%%%%%%%%%%%%%%%%%%%%%%%%%

Since \cref{thm:simplicial_generic_2} clearly holds by combining the following two results (\cref{thm:preserve_strong,thm:simplicial_transverse}) and \cref{thm:main-C1}, it is sufficient to prove \cref{thm:simplicial_transverse}.
\begin{lemma}[\cite{Hamada2019}]\label{thm:preserve_strong}Let $f: \R^m \to \R^\ell$ be a strongly convex mapping.
Then, for any $\pi \in \mathcal{L}(\R^m, \R^\ell)$, the mapping $f+\pi:\R^m\to \R^\ell$ is also strongly convex.
\end{lemma}
\begin{lemma}\label{thm:simplicial_transverse}
Let $f:\R^m\to \R^\ell$ $(m\geq \ell)$ be a $C^r$ mapping $(r\geq 2)$.
Set 
\begin{align*}
\Sigma=\set{\pi\in \mathcal{L}(\R^m,\R^\ell)|\mbox{There exists $x\in \R^m$ such that $\rank d(f+\pi)_x\leq \ell-2$}}.
\end{align*}
If $m-2\ell+4>0$, then for any non-negative real number $s$ satisfying 
\begin{align}\label{eq:generic_s}
s>m\ell-(m-2\ell+4),
\end{align}
the set $\Sigma$ has $s$-dimensional Hausdorff measure zero in $\mathcal{L}(\R^m,\R^\ell)$. 
\end{lemma}
\begin{remark}\label{rem:generic_s}
{\rm We give the following remarks on \cref{thm:simplicial_transverse}.
\begin{enumerate}[(1)]
\item \label{rem:generic_s_1}
In the case $\ell=1$, note that $\Sigma=\varnothing$ and  $m\ell-(m-2\ell+4)=-2$.
Thus, in this case, since the set $\Sigma$ $(=\varnothing)$ has $0$-dimensional Hausdorff measure zero in $\mathcal{L}(\R^m,\R)$, \cref{thm:simplicial_transverse} clearly holds.
\item \label{rem:generic_s_2}
In the case $\ell\geq 2$, since $m\geq \ell$, we have $\codim S^2(\R^m,\R^\ell)=2(m-\ell+2)$.
Thus, the inequality \cref{eq:generic_s} implies that 
\begin{align*}
s>m\ell-(m-2\ell+4)=m\ell+m-2(m-\ell+2)=m\ell+m-\codim S^2(\R^m,\R^\ell).
\end{align*}
\end{enumerate}}
\end{remark}
\begin{proof}[Proof of \cref{thm:simplicial_transverse}]
By \cref{rem:generic_s}~\cref{rem:generic_s_1}, it is sufficient to consider the case $\ell \geq 2$.
As in \cref{rem:generic_s}~\cref{rem:generic_s_2}, we have
\begin{align*}
\codim S^2(\R^m, \R^\ell)=2(m-\ell+2).
\end{align*}
Since $m-2\ell+4>0$, we also have $\codim S^2(\R^m, \R^\ell)>m$.

Let $k$ be an integer satisfying $2\leq k\leq \ell$.
As in \cref{thm:appmain1}, set 
\begin{align*}
\Sigma_k=\set{\pi\in \mathcal{L}(\R^m,\R^\ell)|\mbox{$j^1(f+\pi)$ is not transverse to $S^k(\R^m,\R^\ell)$}}.
\end{align*}
It follows that
\begin{align*}%\label{eq:codim}
m-\codim S^k(\R^m, \R^\ell)\leq m-\codim S^2(\R^m, \R^\ell)<0.
\end{align*}
By \cref{rem:generic_s}~\cref{rem:generic_s_2}, note that \red{the} real number $s$ in \cref{eq:generic_s} satisfies that
\begin{align*}
s>m\ell+m-\codim S^k(\R^m,\R^\ell).
\end{align*}
Since $r\geq 2$, by \cref{thm:appmain1}~\cref{thm:appmain1sub2}, we have the following:
\begin{enumerate}
\item [(a)]
The set $\Sigma_k$ has s-dimensional Hausdorff measure zero in $\mathcal{L}(\R^m, \R^\ell)$.
\item [(b)]
For any $\pi \in \mathcal{L}(\R^m, \R^\ell)\setminus\Sigma_k$, we have $j^1(f+\pi)(\R^m)\cap S^k(\R^m, \R^\ell)=\varnothing$.
\end{enumerate}

By (b), it is clearly seen that $\Sigma=\bigcup_{k=2}^\ell \Sigma_k$.
By (a), the set $\Sigma$ has $s$-dimensional Hausdorff measure zero in $\mathcal{L}(\R^m, \R^\ell)$.
\end{proof}
%%%%%%%%%%%%%%%    
\section*{Acknowledgements}
The author is most grateful to the anonymous reviewers for carefully reading the first manuscript of this paper and for giving invaluable suggestions.
The author was supported by JSPS KAKENHI Grant Number JP21K13786.
He is also grateful for the opportunity to give a talk on \cref{sec:appmulti,sec:simpliciality_generic_proof} at 2021 IMI Joint Use Research Program, Short-term Joint Research “Mathematics for Evolutionary Computation” in Kyushu
University.

%%%%%%%%%%%%%%%%%%%%%%%%%%%%%%%%%%%%%%%%%%%%%%%%%%%%%%%%%%%%%%%%%%%%%%%%%%%%%%%%

\end{document}